\documentclass[11pt]{amsart}  

\usepackage{amssymb,latexsym}
\usepackage[draft=true]{hyperref}
\usepackage{cite} 

 \usepackage{txfonts} 
\usepackage{mathrsfs}     

\theoremstyle{plain}

\newtheorem{lemma}{Lemma}
\newtheorem{corollary}{Corollary}
\newtheorem{proposition}{Proposition}

\theoremstyle{definition}
\newtheorem{definition}{Definition}

\theoremstyle{remark}
\newtheorem{remark}{Remark}

\DeclareMathOperator{\cl}{cl}  
\DeclareMathOperator{\coz}{coz}

\numberwithin{equation}{section}  

\begin{document}
\title[Maximal ideals in rings of  real measurable functions]
{Maximal ideals in rings of  real measurable functions}                               

\author{A. A. Estaji$^{1*}$}
\address{$^{1,2,3}$ Faculty of Mathematics and Computer Sciences, Hakim Sabzevari University, Sabzevar, Iran.}
\thanks{*Corresponding author}
\email{aaestaji@hsu.ac.ir}

\author{A. Mahmoudi Darghadam$^{2}$}
 
\email{m.darghadam@yahoo.com}
\author{H. Yousefpour$^{3}$}
\email{h.yousef79@gmail.com}

\begin{abstract}
Let $ M (X)$ be the ring   of all real measurable functions on a measurable
space $(X, \mathscr{A})$.
In this article, we show that   every ideal of $M(X)$ is a $Z^{\circ}$-ideal.
Also, we give several characterizations of maximal ideals of $M(X)$,
mostly in terms of certain lattice-theoretic properties of $\mathscr{A}$. 
  The notion  of  $T$-measurable space is introduced.
  Next, we show that for every
 measurable space $(X,\mathscr{A})$ there
exists a $T$-measurable space
 $(Y,\mathscr{A}^{\prime})$ such that
 $M(X)\cong  M(Y)$ as rings.
  The notion  of   compact  measurable space  is introduced.
 Next, we prove that
 if $(X, \mathscr{A})$ and  $(Y, \mathfrak{M^{\prime}})$ are two compact
$T$-measurable spaces, then  $X\cong Y$ as measurable spaces if and only if
 $M(X)\cong  M (Y)$ as rings.
\end{abstract}



\maketitle

\section{Introduction} \label{1}
It is well known that $\mathbb{R}^X$ is the collection of all
 real-valued functions on $X$, for every non-empty set $X$
 and this  with the (pointwise) addition and multiplication
 is a reduced commutative ring with identity.
 Let $(X, \mathscr{A})$ be measurable
space and  $M(X, \mathscr{A})$, abbreviated $M(X)$ be
the set of all  real-valued measurable functions on $X$,
then $M(X)$ is a subring of $\mathbb{R}^X$. 
Viertl in \cite{Viertl1977} shows that 
if $X$ is a topological space and $ \mathscr{A}$ is  the set of all  Borel sets  of $X$ then 
 every maximal ideal of $M(X, \mathscr{A})$ is real if and only if 
$ \mathscr{A}$ contains only a
finite number of elements   if and only if 
 every ideal is fixed  of $M(X, \mathscr{A})$. 
 Hager in \cite{Hager1971} shows that if $(X, \mathscr{A})$
is a measurable space then   $M(X)$ is
a regular ring in the sense of von
Neumann (i.e., given $f\in M(X)$ there is $g\in M(X)$ with
$f^2g =f$).  
Azadi et al.  in \cite{AzadiHenriksenMomtahan2009} prove  that if $(X, \mathscr{A})$
is a measurable space then   $M(X)$ is
an $\aleph_0$-self-injective ring and moreover,
if $ \mathscr{A}$  contains all singletons, then
$M(X)$ has an essential socle.
Amini  et al.  in \cite{AminiAminiMomtahanShirdareh2013}
generalized, simultaneously, the ring of real-valued
continuous functions and the ring of real-valued measurable
functions.
Momtahan in \cite{Momtahan2010}
 studied essential ideals, socle, and some related ideals of
  rings of real valued measurable functions.
  He also studied the Goldie dimension of rings
  of measurable functions.
In the  paper  \cite{EstajiMahmoudi} Estaji and Mahmoudi Darghadam investigated  
 rings of  real measurable functions vanishing at infinity on a measurable space.

Let $X$ be any topological space and
let $\mathbb{R}$ be the space of
real numbers with its usual topology.
$C(X)$ is the set of all real-valued
continuous functions with domain $X$ (see \cite{Hewitt1948, Gillman1979}.
The main progress in  the area of rings of real-valued
continuous functions defined over a topological space $X$
 was provided by three historical subject as follows:

\noindent (1).  Completely regular spaces were first discussed by Tychonoff  \cite{Tychonoff1930}.
We recall from   \cite[Theorem 3.9]{Gillman1979} that
for every topological space $X$,
there exists a completely regular  Hausdorff space
$Y$ and a continuous mapping $\tau$ of $X$
onto $Y$, such that the
mapping $g\mapsto go\tau$ is an isomorphism of
$C( Y)$ onto $C(X)$, which the reduction to completely
regular spaces  is due to Stone \cite[P. 460]{Stone1937} and $\check{C}$ech \cite[P. 826]{Cech1937}.

\noindent  (2).
For every $f\in \mathbb{R}^X$, $Z(f):=\{x\in X:f(x)=0\}$
is called the zero-set of $f$.
 An ideal $I$ of $C(X)$ is called fixed if the set
 $\bigcap_{f\in I} Z(f)$ is non-empty;
 otherwise $I$ is called free.
 The maximal fixed ideals of the ring
 $C (X )$ are precisely the sets
 $M_p=\{f\in C(X): f(p)=0\}, $ 
 for every $p\in X$. 
 The Gelfand-Kolmogoroff theorem generalizes this assertion to the
case of arbitrary maximal ideals of $C(X)$ as follows:
(Gelfand-Kolmogoroff). A subset $M$ of $C(X)$
is a maximal ideal of $C(X)$ if and only if there is a unique point
$p\in \beta X$ such that $M$ coincides with the set
 $\{f\in C(X):  p\in \cl_{\beta X} Z(f)\} $ (see \cite{GelfandKolmogoroff1939, GillmanHenriksenJerison1954})

\noindent    (3). Again from \cite[Theorem 3.9]{Gillman1979} we recall that two compact spaces $X$ and $Y$
 are  homeomorphic if and only if their rings
 $C(X)$ and $C(Y)$ are isomorphic.
 This  is due to Gelfand and Kolmogoroff
 (see \cite{GelfandKolmogoroff1939}).

  Notes measurable spaces have been studied
by many authors.
I thought would be of interest to others.
 The present paper is devoted to placing
 these results in a measurable space context.
 We study the three above-mentioned subjects for
 the ring   of all real measurable functions on
 a measurable space.
 The paper is organized as follows.

 Section  \ref{2} of this paper is a prerequisite for the rest of the paper.
The definitions and results of this section are taken from
\cite{Rudin1987}.

In  Section \ref{3},
the notion   of  fixed ideal in  $Z_{\mathscr{A}}$-ideal in
the ring   of all real measurable functions on
 a measurable space are  introduced  and we show that for every
  measurable space $(X, \mathscr{A})$ and every ideal $I$ of $M(X)$,
   $I$ is a $z$-ideal \`a la Mason of $ M(X)$
if and only if  $I$ is a $Z_{\mathscr{A}}$-ideal of $ M(X)$
 if and only if   $I$ is a $Z^{\circ}$-ideal of $ M(X)$ (see Proposition  \ref{M65}). Also, we prove that
 every  ideal in $ M(X)$ is a $z$-ideal  (see Proposition  \ref{M66}).

 In  Section \ref{4}, we show that,
 a subset $M$ of $ M(X)$
is a maximal ideal if and only if  there exists a unique
$J \in \Sigma Id(\mathscr{A})$ such that $M= M^J$,
where $M^J=\{f\in M(X):  \coz(f)\in J\}$
(see Proposition  \ref{max}).
Next,  we  study  the relations between
maximal free ideals  of $ M(X,\mathscr{A})$ and
 prime ideals of $\mathscr{A}$ (see Proposition  \ref{M45-1}).
 Also, we prove that
 for every subset $M$ of   $M(X, \mathscr{A})$,
  $M$ is a fixed maximal ideal of $M(X)$ with
  $\bigcup_{f\in M}  \coz(f)\in \mathscr{A}$
  if and only if there exists a prime element $P$ of
  $\mathscr{A}$ such that   $M=M_P$,
  where $M_P=\{f\in M(X):  \coz(f)\subseteq P\}$
  (see Proposition  \ref{fixmax}).
 Finally, we show that a compact measurable $(X, \mathscr{A})$
  is determined by  fixed maximal ideals of $M(X)$
   (see Proposition  \ref{M95}).

 In  Section \ref{5},
the notion   of  $T$-measurable space
is  introduced  and we show that
for every measurable space $(X,\mathscr{A})$, there
is a $T$-measurable space
 $(Y,\mathscr{A}^{\prime})$
 and a onto function  $\theta:X\rightarrow Y$
  such that
   $\eta: M(Y)\rightarrow M(X)$
 given by  $g\mapsto g\circ\theta$
 is an isomorphism (see Proposition \ref{M290}).
  In Proposition \ref{M105}, we give an algebraic characterization
 $T$-measurable spaces in terms of maximal ideals and this implies that
  if  $(X, \mathscr{A})$ is a $T$-measurable space
then for every element $P$ of   $\mathscr{A}$,
  $P$ is a prime element of
  $\mathscr{A}$   if and only if \, $|X\setminus P|=1$.
(see Corollary  \ref{X-P=1}).
Also, we show that a measurable space $(X, \mathscr{A})$ is a  $T$-measurable space
if and only if
for every prime ideal $P$ in $ M(X)$, $|\bigcap_{f\in P}Z(f)|\leq 1$
  (see Proposition \ref{M120}).

 In  Section \ref{6},
 we prove that for every compact measurable space $(X,\mathscr{A})$ there
is a   compact $T$-measurable  space
 $(Y,\mathscr{A}^{\prime})$
  such that  $M(X)\cong  M(Y) $
  as rings (see Corollary  \ref{M295}).
  Also, for every   compact $T$-measurable  space $ (X, \mathscr{A})$, we show that
  $X\cong \max(M(X))$  as measure spaces,
 also, if $(X, \mathscr{A})$ and  $(Y, \mathfrak{M^{\prime}})$ are two compact
$T$-measurable spaces, then  $X\cong Y$ as measurable spaces if and only if
 $M(X)\cong  M (Y)$ as rings (see
Proposition \ref{M215} and
Corollary  \ref{M220}).
\section{Preliminaries} \label{2}
Here, we recall some definitions and results from the literature on measurable spaces
and partially ordered sets. For
further information see \cite{Rudin1987} on measurable-theoretic concepts and \cite{EbrahimiMahmoudi1996, PicadoPultr2012} on lattice-theoretic concepts.

\subsection{measurable spaces} 

Let $X$ be a nonempty set. 
Let us   recall some general notation from \cite{Rudin1987}.
 A collection $\mathscr{A}$ of subsets of a set $X$
is said to be a {\it $\sigma$-algebra} in $X$
if $\mathscr{A}$ has the following three properties:
\begin{enumerate}
\item[{\rm (i)}]  $X\in \mathscr{A}$.
\item[{\rm (ii)}]  If $A \in \mathscr{A}$, then $A^c \in \mathscr{A}$, where $A^c$ is the complement of $A$ relative to $X$.
\item[{\rm (iii)}]  If $\{A_n\}_{n\in\mathbb N} \subseteq  \mathscr{A}$, then $\bigcup_{n\in\mathbb N} A_n \in \mathscr{A}$.
\end{enumerate}
 If $\mathscr{A}$ is a $\sigma$-algebra in $X$,
 then $X$  or, for clarity  $(X,  \mathscr{A} )$  is called a  {\it measurable space}, and
the members of $\mathscr{A}$ are called
the measurable sets in $X$.
If $X$ is a measurable space, $Y$ is a topological space,
 and $f$ is a mapping of $X$ into $Y$,
then $f$ is said to be  {\it measurable} provided
that $f^{-1}(V)$ is a measurable set in $X$
for every open set $V$ in $Y$.
If $X$ is a measurable space, then
the set of all measurable maps from $X$ into $\mathbb R$ is denoted $M(X)$,
and the members of $M(X)$ are called the  {\it real measurable functions} on $X$,
where $\mathbb R$ denotes the set of all real numbers with the ordinary topology.

We recall from \cite[1.10 Theorem]{Rudin1987} that
  if $\mathcal{A}$ is any collection of subsets
  of $X$, there exists a smallest
$\sigma$-algebra $\mathscr{A}^*$ in $ X$
such that $\mathcal{A}\subseteq \mathscr{A}^*$.
This $\mathscr{A}^*$ is  called the  {\it $\sigma$-algebra generated by  $\mathcal{A}$}  and
it is denoted by $< \mathcal{A} >$.
 Since  the intersection of any  family of $\sigma$-algebras in $X$ is a $\sigma$-algebra in $X$,
 we conclude that
   $<\mathcal{A}>$ is  intersection of
 the family of all $\sigma$-algebras $\mathscr{A}$ in $X$ which contain $\mathcal{A}$.
 Hence  $<\mathcal{A}>=<\mathcal{A}^{c}>=<\mathcal{A}\cup
  \mathcal{A}^{c}>$,
  where $\mathcal{A}^{c}:=\{A^{c}:
 A \in \mathcal{A}\}$.
 Also, if  $\mathcal{A}, \mathcal{B}\subseteq \mathcal{P}(X)$ with $\mathcal{A}\subseteq
\mathcal{B}$ , then $<\mathcal{A}>\subseteq<\mathcal{B}>$.

The set $ M(X)$ of all real measurable functions on a measurable
space $(X, \mathscr{A})$ will be provided with an algebraic structure and an
order structure.
Since their definitions do not involve measurity, we begin by imposing
these structures on the collection $\mathbb{R}^X$ of all functions from $X$ into
the set $\mathbb R$ of real numbers. Addition,  multiplication, joint, and meet in $\mathbb{R}^X$ are defined by
the formulas
$(f + g)(x) = f(x) + g(x)$,
 $(fg)(x) = f(x)g(x)$,
 $(f\vee g)(x)=\max\{f(x) , g(x)\}$,
 and $(f\wedge g)(x)=\min\{f(x) , g(x)\}$.
It is obvious that
$(\mathbb{R}^X; +,.,\vee, \wedge)$ is an $f$-ring,
this conclusion is immediate consequence of
the corresponding statements about the field $\mathbb R$.
Also,
$(M(X); +,.,\vee, \wedge)$ is an sub-$f$-ring of $\mathbb{R}^X$.

\subsection{partially ordered sets}
Let us   recall some general notation from \cite{EbrahimiMahmoudi1996, PicadoPultr2012}. 
A poset $L$ is a {\it lattice} if and only if  for every $a$ and  $b$ in $L$ both $\sup\{a, b\}$ and
 $\inf\{a, b\}$ exist (in $L$).
 For subset $X$ of a poset  $ L$ and $x\in L$ we write:
\begin{enumerate}
\item $\downarrow X=\{y\in L: y\leq x \mbox{ for some } x\in X\}$.
\item $\uparrow X=\{y\in L: y\geq x \mbox{ for some } x\in X\}$.
\item $\downarrow x=\downarrow \{x\}$.
\item$\uparrow x=\uparrow \{x\}$.
\end{enumerate}
 A nonempty subset $J$ of a   lattice   $L$ is called an  {\it ideal} of $L$ if
 $x \vee y \in J$ and  $\downarrow x\subseteq J$, for all $x, y \in J$.
  A nonempty subset $F$ of a   lattice    $L$ is called a  {\it filter} of $L$ if  $x \wedge y \in F$ and $\uparrow x \subseteq F$, for all $x, y \in F$.
  We say an element $a$ of a   lattice    $L$ is a
 {\it top element}  ({\it bottom element}) of $L$
 if $x \leq a$ ($a\leq x$), for all $x \in X$.
 We denote the top element and the bottom element of  a   lattice   $L$ by $\top$ and $\bot$ respectively.
 A lattice  $L$ is said to be {\it bounded} if there exist a top  element   and a bottom element  in the lattice.
 An element $a$ of a bounded lattice $L $ is said to be {\it compact}
if $a=\bigvee S$, $S\subseteq L$, implies $a=\bigvee T$ for some finite subset
$T$ of $ S$. A  bounded lattice  $L$ is said to be {\it compact} whenever its  top element
$\top$ is compact.
A lattice $L$ is said to be distributive lattice if the binary operations
$\vee$ and $\wedge$
holds distributive property, i.e.; for any $x, y , z\in L$,
$x\vee(y\wedge z)=(x\vee y)\wedge (x\vee z)$
and
$x\wedge(y \vee z)=(x\wedge y)\vee (x\wedge z)$.
An element $p$ of a bounded  distributive lattice $L$ is said
 to be {\it prime} if $p<\top$ and $a\wedge
b\leq p$ implies $a\leq p$ or $b\leq p$.
An element $m$  of a bounded lattice $ L$ is
said to be {\it maximal} (or {\it dual atom})  if $m<\top$ and $m\leq x\leq \top$ implies $m=x$
or $x=\top$.
As it is well known, every maximal element
in  a bounded  distributive lattice  is prime.
We write $\Sigma L$
and $Max(L)$ for the set of all prime elements and
maximal elements  of $L$ ,  respectively.
A $\sigma$-frame is a
lattice $L$ with countable joins $\bigvee_n$, finite meets $\wedge$, top $\top$, bottom $\bot$ and satisfying
$x\wedge \bigvee_n x_n = \bigvee_n x\wedge x_n$, for $n\in J$, a countable index set, $x,x_n\in L$.
 A {\it frame} 
  is a complete lattice $L$
in which the distributive law
 $ x \wedge \bigvee S = \bigvee_{s \in S }(x \wedge s)$
holds for all  $ x \in L$ and $ S \subseteq L$.
 The frame of open subsets of a topological space $X$ is denoted by $\mathfrak O X$‎.

If $\mathscr{A}$ is a $\sigma$-algebra in $X$, then the following statements hold.
\begin{enumerate}
\item[{\rm (1)}] $( \mathscr{A},\subseteq)$ is a Boolean algebra.
\item[{\rm (2)}]   $( \mathscr{A},\subseteq)$ is a $\sigma$-frame.
\end{enumerate}
\section{Filters in $\sigma$-algebras}  \label{3}
  Consider $f\in \mathbb{R}^X$,
the set $f^{-1}(0)$ will be called the zero-set of $f$. We shall find it convenient
to denote this set by $Z(f)$, or, for clarity, by $Z_X(f)$:
$$
Z(f) = Z_X(f) = \{x \in X:f(x) = 0\} .
$$
Any set that is a zero-set of some function in $\mathbb{R}^X$ is called a zero-set
in $X$.   For any subset $A$ from  $  M(X)$, we write  $Z_{\mathscr{A}}[A]=\{Z(f):f\in A\}$
and we put $Z_{\mathscr{A}}[X]=Z_{\mathscr{A}}[  M(X)]$.

For every $f, g\in  \mathbb{R}^X$,  we have
\begin{enumerate}
\item[{\rm (1)}]  $Z(f)=Z(|f|)=Z(f^n)$,  every $n\in\mathbb N$.
\item[{\rm (2)}]  $Z(fg)=Z(f)\cup Z(g)$.
\item[{\rm (3)}]  $Z(f^2+g^2)=Z(|f|+|g|)=Z(f)\cap Z(g)$.
\end{enumerate}


\begin{remark} \label{chi}
Let $(X, \mathscr{A})$ be a measurable space and $A\in \mathscr{A}$,
then the characteristic function $‎‎\chi‎_{_A}:X\rightarrow \mathbb R$
is a real measurable function on $X$ with $Z(\chi‎_{_A})=A^c$.
\end{remark}
 \begin{proposition}
   \label{M15}
If $(X, \mathscr{A})$ is a measurable
space, then $Z_{\mathscr{A}}[X]= \mathscr{A}$.
\end{proposition}
\begin{proof}
 Let $f$  be a real measurable functions on the measurable
space $(X, \mathscr{A})$, then
$Z(f)=f^{-1}[0, +\infty)\cap f^{-1}(-\infty, 0]\in \mathscr{A}$.
Therefore,   $Z_{\mathscr{A}}[X]\subseteq  \mathscr{A}$. By Remark \ref{chi},
$ \mathscr{A}\subseteq Z_{\mathscr{A}}[X]$ and so $Z_{\mathscr{A}}[X]= \mathscr{A}$.
\end{proof}

Let $(X, \mathscr{A})$ be a measurable space.
Then we have
 $\bigcap_{n\in \mathbb N}Z(f_n)\in Z_{\mathscr{A}}[X]$  and
 $\bigcup_{n\in \mathbb N}Z(f_n)\in Z_{\mathscr{A}}[X]$,
for every  $\{f_n\}_{n\in \mathbb N}\subseteq  M(X)$.
 Therefore,  $Z_{\mathscr{A}}[X]$ is a $\sigma$-frame.

\begin{proposition} \label{unit}
 Let $f$  be a real measurable functions on a measurable
space $(X, \mathscr{A})$.
 The element $f$ is a unit element of $ M(X)$ if and only if $Z(f) = \emptyset$.
\end{proposition}
\begin{proof}
\textit{Necessity.}
By hypothesis, there is a $g\in  M(X)$ such that $fg=1$, then
$Z(f)\cup Z(g)=Z(fg)=Z(1)=\emptyset$, which follows that $Z(f) = \emptyset$.

\textit{Sufficiency.}
We define $g:X\rightarrow \mathbb R$ given by $g(x)=\frac{1}{f(x)}$, for every $x\in X$.
Since $Z(f) = \emptyset$, we conclude that $g\in  \mathbb{R}^X$. We have
$$
g^{-1}(r, +\infty)
=\big(\{x\in X: rf(x)<1\}\cap f^{-1}(0, +\infty)
\big)
\cup
\big(\{  x\in X: rf(x)>1\}\cap f^{-1}(-\infty,0)
\big),
$$
for every $r\in\mathbb R$. Since ${\bf r}f\in  M(X)$,
we infer that $g^{-1}(r, +\infty)\in \mathscr{A}$,
for every $r\in\mathbb R$.
Therefore, $g\in   M(X)$ and $fg=\mathbf{1}$.
Hence  $f$ is a unit element of $ M(X)$.
\end{proof}



\begin{definition}  \label{M25}
Let $(X, \mathscr{A})$ be a measurable space.
A proper filter of $\mathscr{A}$ is called a   
{\it $Z_{\mathscr{A}}$-filter} on $X$.
\end{definition}
 In the following proposition, we study
relations between proper ideals and    $Z_{\mathscr{A}}$-filters.
\begin{proposition}  \label{M30}
Let $(X, \mathscr{A})$ be a measurable space.
In $ M(X)$, the following statements hold.
\begin{enumerate}
\item[{\rm (1)}]  If $I$ is a proper ideal in $ M(X)$, then the family
$Z[I] = \{Z(f)\mid f \in I\}$
is a $Z_{\mathscr{A}}$-filter on $X$.

\item[{\rm (2)}]  If $  \mathcal{F}$ is a $Z_{\mathscr{A}}$-filter on $X$, then the family
 $Z^{-1}[\mathcal{F}] = \{f\mid Z(f) \in \mathcal{F}\}$
is  a proper  ideal in $ M(X)$.
\end{enumerate}
\end{proposition}
\begin{proof}
 (1).  Consider $f,g\in I$ and $h\in  M(X)$.
By hypothesis, $f^2 + g^2, fh\in I$, then
$Z(f)\cap Z(g)=Z(f^2+g^2)\in Z[I]$ and if $Z(f)\subseteq Z(h)$, then $Z(h)=Z(f)\cup Z(h)=Z(fh)\in Z[I]$. Also, since $I$ contains no unit, we conclude from  Proposition \ref{unit}  that $\emptyset \notin Z[I]$.

(2).  Let $J = Z^{-1}[\mathcal{F}]$. By Definition $Z_{\mathscr{A}}$-filter and Proposition \ref{unit}, $J$ contains no unit. Let $f, g \in J$,
and let $h \in  M(X)$. Then
\[
Z(f - g)
= Z(f +(- g))
\geq Z(f)\cap Z(- g))
\geq Z(f)\cap Z(g)\in \mathcal{F}
\]
and hence $Z(f - g)\in \mathcal{F}$, by Definition $Z_{\mathscr{A}}$-filter. Therefore $f - g\in Z^{-1}[\mathcal{F}]$. Moreover,
 $$Z(hf)= Z(h)\cup Z(h)\supseteq Z(f)\in \mathcal{F},$$
and hence $Z(fh)\in \mathcal{F}$, by Definition $Z_{\mathscr{A}}$-filter. Therefore $fh\in Z^{-1}[\mathcal{F}]$.
This completes
the proof that $J$ is a proper ideal in $ M(X)$.
\end{proof}
Let $(X, \mathscr{A})$ be a measurable space.
A $Z_{\mathscr{A}}$-filter $\mathcal F$ on a set $X$ is said to be an {\it $Z_{\mathscr{A}}$-ultrafilter} if it is maximal (with respect to inclusion) in the family of all $Z_{\mathscr{A}}$-filters on $X$.

In the following proposition, we study
relations between maximal ideals and    $Z_{\mathscr{A}}$-ultrafilters.
\begin{proposition}\label{M35}
Let $(X, \mathscr{A})$ be a measurable space.
In $ M(X)$,  the following statements hold.
\begin{enumerate}
\item[{\rm (1)}]  If $M$ is a maximal ideal in $ M(X)$, then $Z[M]$ is a $Z_{\mathscr{A}}$-ultrafilter on $X$.
\item[{\rm (2)}]  If $ \mathcal F$ is a $Z_{\mathscr{A}}$-ultrafilter on $X$, then $Z^{-1}[\mathcal F]$ is a maximal ideal in $ M(X)$.

The mapping $Z$ is one-one from the set of all maximal ideals in $ M(X)$ onto
the set of all $Z_{\mathscr{A}}$-ultrafilters on $X$.
\end{enumerate}
\end{proposition}
\begin{proof}
Since $ Z_{\mathscr{A}} $ and $Z_{\mathscr{A}}^{-1}$ preserve inclusion, the result follows at once from Proposition \ref{M30}.
\end{proof}
Let $ A\in   \mathscr{A}$ and $ \mathcal{F}  \subseteq\mathscr{A}$, we say  $ A $ meets $ \mathcal{F} $
if and only if  $ A\cap B\neq \emptyset$, for all $ B\in\mathcal{F} $.
It is evident that
\begin{enumerate}
\item A   $Z_{\mathscr{A}}$-filter $ \mathcal F $  on $X$ is a $Z_{\mathscr{A}}$-ultrafilter if and only if
 $ A $ meets $ \mathcal F $ implies  $ A\in \mathcal F $, for every $ A\in  \mathscr{A}$.
\item
 If $ \mathcal F $ and $ \mathcal G $ are disjoint $Z_{\mathscr{A}}$-ultrafilter  on $X$,
 then there is elements $ A\in\mathcal F $ and $B\in \mathcal G $ such that $A\cap B=\emptyset $.
\item If $\{ \mathcal F_i\}_{i \in  I }$ is a nonempty collection of $Z_{\mathscr{A}}$-filters   on $X$, then  $ \bigcap_{i \in  I}\mathcal F_i $ is a $Z_{\mathscr{A}}$-filter   on $X$.
\item     Every $Z_{\mathscr{A}}$-filter  on $X$ is contained in  a $Z_{\mathscr{A}}$-ultrafilter  on $X$ .
\end{enumerate}
\begin{proposition}\label{M40}
 Let $M$ be a maximal ideal in $  M(X)$.
 If $Z(f)$ meets every member of $Z[M]$, then $f \in M$.
\end{proposition}
\begin{proof}
 The set $Z[M]$ is a $Z_{\mathscr{A}}$-ultrafilter on $X$,
 by Proposition \ref{M35},
 and so, if $Z(f)$ meets every member of $Z[M]$,
 then $ Z(f)\in Z[M] $.
  Therefore  $ f\in Z^{-1}[Z[M]] ${\em ;}
  moreover, $ M\subseteq Z^{-1}[Z[M]] $,
  and $M$ is a maximal ideal, so that $ f\in M=Z^{-1}[Z[M]] $.
\end{proof}
\begin{proposition} \label{M45}
Let $(X, \mathscr{A})$ be a measurable space.
For every $p\in X$,
$$
M_p:=\{f\in   M(X) : f(p)=0\}
$$
is a maximal ideal in $ M(X)$.
\end{proposition}
\begin{proof}
By Proposition \ref{unit}, it is clear that $M_p$ is a proper ideal in  $ M(X)$.
Consider $f\in  M(X)\setminus M_p$ with $f(p)=r\not=0$.
From $\mathbf 1-\mathbf{\frac 1r}f\in   M(X)$ and
 $(\mathbf 1-\mathbf{\frac 1r}f)(p)=0$, we infer that
  $\mathbf 1-\mathbf{\frac 1r}f\in M_p$, which follows that $M_p$ is a maximal ideal in $ M(X)$.
\end{proof}
 Recall the notion of a $z$-ideal
of a ring $A$ as was introduced by Mason in \cite{Mason1973}. In lattice theory this notion is known as ``$z$-ideals \`a la Mason''.
Denote by $Max(A)$ the set of all maximal ideals of a ring $A$. For $a \in A$, let
\begin{center}
$ \mathfrak M(a) = \{M \in Max(A) \mid a \in M\}.$
\end{center}
An ideal $I$ of a ring $A$ is called a
  \textbf{$z$-ideal  \`a la Mason}  if whenever $\mathfrak M(a) \subseteq \mathfrak M(b)$
and $a\in I$, then $b\in I$.

\begin{lemma}  \label{M50}
Let $(X, \mathscr{A})$ be a measurable space.
In $ M(X)$,  the following statements are equivalent, for every $f, g\in  M(X)$.
\begin{enumerate}
\item[{\rm (1)}]  $ \mathfrak M(f) \subseteq \mathfrak M(g) $.
\item[{\rm (2)}]  $Z(f)\subseteq Z(g)$.
\item[{\rm (3)}]  $Ann(f)\subseteq Ann(g)$.
\end{enumerate}
\end{lemma}
\begin{proof}
(1)$\Rightarrow$(2).
We have
$$
p\in Z(f)
\Rightarrow f\in M_p\in \mathfrak M(f) \subseteq \mathfrak M(g)
\Rightarrow g\in M_p
\Rightarrow p\in Z(g).
$$

(2)$\Rightarrow$(3).
We have
$$
h\in Ann(f)
\Rightarrow X=Z(hf)=Z(h)\cup Z(f)
\subseteq Z(h)\cup Z(g)=Z(hg)
\Rightarrow h\in Ann(g).
$$

(3)$\Rightarrow$(1).
Consider $M\in  \mathfrak M(f)$.
Since
$$Z(\chi_{_{Z(f)}}f)
=Z(\chi_{_{Z(f)}})\cup Z(f)
=(X\setminus Z(f))\cup Z(f)
=X,
$$
we conclude that $\chi_{_{Z(f)}}\in Ann(f)\subseteq Ann(g)$, which follows that $\chi_{_{Z(f)}}g=0\in M.$
Since $Z(\chi_{_{Z(f)}}^2+f^2)=\emptyset$,
we conclude from Proposition \ref{unit} that
$g\in M$, it follows that $M\in  \mathfrak M(g)$.
Therefore $ \mathfrak M(f) \subseteq \mathfrak M(g) $.
\end{proof}
For each $a \in R$ let $P_a$ be the intersection of all minimal prime ideals
containing $a$ and by convention, we put the
intersection of an empty set of ideals equal to $R$.
We recall from \cite{AzarpanahKaramzadehRezai2000} that a proper ideal $I$ of a ring $R$ is called a $Z^{\circ}$-ideal if for each $a \in I$
we have $P_a \in I$. Also,

\begin{proposition} \label{M55}
{\rm \cite{AzarpanahKaramzadehRezai2000}}
Let $R$ be a reduced ring and $I$ be a proper ideal in $R$,
then the following are equivalent.
  \begin{enumerate}
\item[{\rm (1)}]  I is a $Z^{\circ}$-ideal in $R$.
\item[{\rm (2)}]  $P_a = P_b$ and $a\in I$, imply that $b\in I.$
\item[{\rm (3)}]  $Ann(a) = Ann(b)$  and $a\in I$, imply that $b\in I.$
\item[{\rm (4)}]  $a \in I$ implies that $Ann(Ann(a))\subseteq  I.$
\end{enumerate}
\end{proposition}

\begin{definition} \label{M60}
An ideal $I$ of a ring $ M(X)$ is called a
  \textit{ $Z_{\mathscr{A}}$-ideal}   if whenever $Z(f)\subseteq Z(g) $, $f\in I$ and $g\in  M(X)$, then $g\in I$.
\end{definition}
\begin{proposition} \label{M65}
Let $(X, \mathscr{A})$ be a measurable space.
If $I$ is an ideal in $ M(X)$, then the following statements are equivalent.
\begin{enumerate}
\item[{\rm (1)}]   $I$ is a $z$-ideal \`a la Mason of $ M(X)$.
\item[{\rm (2)}] $I$ is a $Z_{\mathscr{A}}$-ideal of $ M(X)$.
\item[{\rm (3)}] $I$ is a $Z^{\circ}$-ideal of $ M(X)$.
\end{enumerate}
\end{proposition}
\begin{proof}
By Lemma \ref{M50} and Proposition \ref{M55}, it is evident.
\end{proof}
\begin{proposition} \label{M66}
Let $(X, \mathscr{A})$ be a measurable space. Every  ideal in $ M(X)$ is a
 $z$-ideal.
\end{proposition}
\begin{proof}
Let $I$ be an  ideal of $ M(X)$.
Suppose that $f,g\in  M(X)$ with
 $Z(f)\subseteq Z(g) $ and  $f\in I$.
 We define $h:X\to \mathbb{R}$ by $h(x)=\frac{1}{f(x)}$, if $x\in  \coz(f)$
  and $0$ otherwise. Then $h\in M(X)$ and
  $g=ghf\in I$. This completes the proof, by Proposition \ref{M65}.
\end{proof}
The following proposition shows that
 the primeity
of a ideal  in $M(X)$   coincides with its
semiprimeity.
\begin{proposition}   \label{M110}
Let $(X, \mathscr{A})$ be a measurable space.
Let $I$ be a  proper ideal in $ M(X)$.
Then the following statements are equivalent.
\begin{enumerate}
\item[{\rm (1)}]  $I$ is a prime ideal.
\item[{\rm (2)}]  $I$ contains a prime ideal.
\item[{\rm (3)}]  For every  $f,g \in  M(X)$, if $fg= \mathbf{0}$, then $f\in I$ or $g\in I$.
\item[{\rm (4)}]  For every $f \in  M(X)$, there is a zero set belonging to $Z[I]$ on which $f$ does not change sign.
\end{enumerate}
\end{proposition}
\begin{proof}
(1)$\Rightarrow$(2) and (2)$\Rightarrow$(3).
 Trivial.

(3)$\Rightarrow$(4).
 Observe that  for
every $f\in M(X)$,
$$(f \vee \mathbf 0)(f \wedge \mathbf 0)=f^+(-f^-) = \mathbf 0.$$
 Then, by hypothesis, either $f \vee \mathbf 0$ or $f \wedge \mathbf 0$ is  in $I$ , and hence $Z(f \vee \mathbf 0)$ or $Z(f \wedge \mathbf 0)$ is  in $Z[I]$, and  however,  $f$ does not change sign on them, since
 $$  Z(f \wedge \mathbf 0)\cap f^{-1}(-\infty,0)=f^{-1}[0,+\infty)\cap f^{-1}(-\infty,0)=\emptyset $$
 and
 $$ Z(f \vee \mathbf 0)\cap f^{-1}(0,+\infty)=f^{-1}(-\infty,0]\cap f^{-1}(0,+\infty)=\emptyset.$$

 (4)$\Rightarrow$(1).
 Given $gh \in I$, consider the function $\vert g\vert- \vert h\vert$.
By hypothesis, there is a zero-element $Z$ of $Z[I]$ on which $\vert g\vert- \vert h\vert$ is
non-negative, say,
$ Z\cap(\vert g\vert- \vert h\vert)^{-1}(-\infty,0)=\emptyset $. Then there is a $f\in I$  such that $Z=Z(f)$, it follows that
 $Z(f)\cap Z(g)\subseteq Z(h)$.
 From
 $$
 Z((hg)^2+f^2)
 =Z(hg)\cap Z(f)
= [Z(h)\cap Z(f)]\cup[Z(g)\cap Z(f)]
 \subseteq Z(h)
 $$
and $ (hg)^2+f^2\in I$,
 we conclude that $h\in I$, since, by Propositions \ref{M65} and \ref{M66},  $I$ is the $Z_{\mathscr{A}}$-ideal in $ M(X)$. Thus, $I$ is prime.
\end{proof}

\section{Maximal ideals in $M(X)$ } \label{4}
Let $(X, \mathscr{A})$ be a measurable space.
For each $I \in Id(\mathscr{A})$, the ideals $M^I$ is
 defined by
$$M^I = \{f\in  M(X) :  \coz(f)\in I\}.$$
\begin{lemma}  \label{I=J}
Let $(X, \mathscr{A})$ be a measurable space.
 For each  $I, J \in Id(\mathscr{A})$,
  $M^I =M^J$ if and only if
  $I=J$.
\end{lemma}
\begin{proof}
\textit{Necessity.}
If $a\in I$ then $\chi_a\in M^I $, which follows that
$a= \coz(\chi_a)\in J$. Therefore, $I=J$.

\textit{Sufficiency.}  It is clear.
\end{proof}

\begin{remark} \label{prime=max}
Let $(X, \mathscr{A})$ be a measurable space.
Consider $I \in \Sigma Id(\mathscr{A})$ and
$I\subsetneq J\in  Id(\mathscr{A})$.
Then there exists $a\in J\setminus I$ and since
$a\wedge a'=\bot\in I \in \Sigma Id(\mathscr{A})$,
we conclude that $a'\in I\subseteq J$,
which follows that $\top=a\vee a'\in J$ and so,
$J= \mathscr{A}$. Therefore,
$$\Sigma Id(\mathscr{A})=Max(Id(\mathscr{A})).$$
\end{remark}
In the following proposition, we
 investigate  the relations between
maximal ideals    of $ M(X)$ and
maximal ideals    of $\mathscr{A}$.
\begin{proposition} \label{max} max
Let $(X, \mathscr{A})$ be a measurable frame.
A subset $M$ of $ M(X)$
is a maximal ideal if and only if  there exists a unique
$J \in \Sigma Id(\mathscr{A})$ such that $M= M^J$.
\end{proposition}
\begin{proof} \textit{Necessity.}
Let $M$ be a maximal ideal   of $ M(X)$.
We set $I=\{a\in \mathscr{A} :  a\leq  \coz(f) \mbox{ for some $f\in M$}\}$. By Proposition \ref{unit}, $\mathscr{A}\not=I\in Id(\mathscr{A})$ and
since $Id(\mathscr{A})$ is a compact frame,
we conclude that $J \in \Sigma Id(\mathscr{A})$ such that
$I\subseteq J$.
From $\top\not\in J$, we infer that
 $M\subseteq M^J\not=  M(X)$,
 and  in view of
the maximality of $M$ we must have $M= M^J$.
By Lemma \ref{I=J},  there exists a unique
$J \in \Sigma Id(\mathscr{A})$ such that $M= M^J$.

\textit{Sufficiency.}
Consider $ J \in \Sigma Id(\mathscr{A})$ and
$Q\in Id(  M(X))$ with
$M^J\subsetneq Q$.
Then there exists $f\in Q\setminus M^J$.
From $ \coz(f)\not\in J$, we infer from Remark \ref{prime=max} that
there exists $a\in J$ such that
$$ \coz(\chi_a + f^2)= \coz(\chi_a)\vee  \coz(f)=a\vee  \coz(f)=\top,$$
which follows that $\chi_a + f^2\in Q$ is a unit element of
$ M(X)$ and so,
$Q= M(X)$.
Therefore,  $M^J$ is a maximal ideal   of $ M(X)$.
\end{proof}
\begin{definition} \label{M90}
Let $I$ be any ideal in $ M(X)$. If $\bigcap_{f\in I}Z(f)$ is nonempty, we
call $I$ a {\it fixed ideal}; if $\bigcap_{f\in I}Z(f)=\emptyset$, then $I$ is a  {\it free ideal}.
 Also, if $\mathcal K\subseteq \mathscr{A}$  with $\bigcap\mathcal K$ is nonempty, we
call $\mathcal K$ a  {\it fixed subset} of $\mathscr{A}$; if $\bigcap\mathcal K=\emptyset$, then $\mathcal K$ is a  {\it free subset} of $\mathscr{A}$.
\end{definition}
In the following proposition, we
 investigate  the relations between
 fixed maximal ideals    of $ M(X)$ and
prime ideals    of $\mathscr{A}$.
\begin{proposition} \label{M45-1}  M45-1
Let $(X, \mathscr{A})$ be a measurable space.
 For every subset $M$ of   $M(X)$,
  $M$ is a fixed maximal ideal of $M(X)$
  if and only if there exists a prime ideal $P$ of
  $\mathscr{A}$ such that $\bigcup P\subsetneq X$ and $M=M^P.$
\end{proposition}
\begin{proof}
\textit{Necessity.}
Let $M$ be a fixed maximal ideal of $M(X)$.
Then, by Proposition \ref{max}, $M = M^P$ for some  $P$ of $\Sigma Id(\mathscr{A})$.
 Since   for every $A\in P$, $\chi_A\in M^{P}$,
 we infer that $ \bigcup P=\bigcup_{f\in M^P}  \coz(f)\subsetneq X$.

\textit{Sufficiency.}
Consider $P\in \Sigma Id(\mathscr{A})$ with $\bigcup P\subsetneq X$.
Then, by   Proposition \ref{max},
$ M^{ P}$
is a maximal ideal in $M(X)$.
Since   for every $A\in P$, $\chi_A\in M^{P}$,
we conclude that $\bigcup_{f\in M^P}  \coz(f)=\bigcup P\subsetneq X$,
which follows that $M^P$  is a fixed maximal ideal of $M(X)$.
\end{proof}
Let $(X, \mathscr{A})$ be a measurable space.
For each $A \in \mathscr{A}$ with $A \subsetneq X$,
define the subset $M_A$ of $M(X)$ by
$$
M_A:=\{f\in     M(X) :  \coz(f)\subseteq A\}
$$

In the following proposition, we
 investigate  the relations between
 fixed maximal ideals    of $ M(X)$ and
prime elements    of $\mathscr{A}$.
\begin{proposition} \label{fixmax}
Let $(X, \mathscr{A})$ be a measurable space.
 For every subset $M$ of   $M(X)$,
  $M$ is a fixed maximal ideal of $M(X)$ with
  $\bigcup_{f\in M}  \coz(f)\in \mathscr{A}$
  if and only if there exists a prime element $P$ of
  $\mathscr{A}$ such that   $M=M_P.$
\end{proposition}
\begin{proof}
\textit{Necessity.}
We claim that $P:=\bigcup_{f\in M}  \coz(f)$ is a
 prime element   of $\mathscr{A}$.
 If not, there exist $V,W\in  \mathscr{A}$ such that
 $V\cap W\subseteq P$,  $V\not \subseteq P$ and
 $V\not\subseteq P$, then $\chi_{_{X\setminus V}}, \chi_{_{X\setminus W}}\in M$ and this implies
 $X=(X\setminus V)\cup (X\setminus W)\cup (V\cap W) =P$,
 but this is a contradiction to the fact that
 $M$ is a fixed maximal ideal of $M(X)$, which proves the claim.
 By the maximality of $M$, we have $M=M_P$, since
 $M\subseteq M_P\subsetneq M(X)$.

\textit{Sufficiency.}
Consider $P\in \Sigma \mathscr{A}$.
From $\downarrow P\in  \Sigma Id(\mathscr{A})$,
we conclude from Proposition \ref{max} that
$M_P=M^{\downarrow P}$
is a maximal ideal in $M(X)$.
Since $\chi_{P}\in M^{\downarrow P}$, we conclude that
  $\bigcup_{f\in M^{\downarrow P}}  \coz(f)= P\subsetneq X$,
  which follows that $M_P$
is a fixed maximal ideal in $M(X)$,
by   Proposition \ref{M45-1}.
\end{proof}
As an immediate consequence we now have the following corollary.
\begin{corollary} \label{fixprim}
Let $(X, \mathscr{A})$ be a measurable space.
 For every subset $M$ of   $M(X)$  with
  $\bigcup_{f\in M} coz(f)\in \mathscr{A}$,
  $M$ is a fixed prime ideal of $M(X)$  
  if and only if   $M$ is 
   a fixed maximal ideal of $M(X)$. 
\end{corollary}


A measurable space $(X, \mathscr{A})$ is called a {\it compact measurable space} if $\mathscr{A}$ is a compact   lattice.

\begin{definition} \label{M75}
Let $\mathcal K$ be a nonempty family of
sets. $\mathcal K$ is said to have the  {\it finite intersection property}
provided that the intersection of any finite  number of
members of $\mathcal K$ is nonempty.
\end{definition}
\begin{remark} \label{M80}
Let $(X, \mathscr{A})$ be a measurable space.
It is evident that every subfamily of $\mathscr{A}$ with the
finite intersection property is contained in some $Z_{\mathscr{A}}$-ultrafilter.
\end{remark}

\begin{proposition}  \label{M85}
Let $(X, \mathscr{A})$ be a measurable space. Then    $(X, \mathscr{A})$ is a compact measurable space 
if and only if for every $\mathcal K\subseteq \mathscr{A}$, if  $\mathcal K$  have the finite intersection property, then $\bigcap\mathcal{K}\not=\emptyset.$
\end{proposition}
\begin{proof}
\textit{Necessity.}   Let $\mathcal K$  be any family of
measurable subsets of $X$ with the finite intersection property.
If $\bigcap\mathcal{K} =\emptyset$,
then $\bigcup_{A\in \mathcal{K}}A^c=X$ and from $\{A^c : A\in \mathcal{K}\}\subseteq \mathscr{A}$,
we conclude that there exist $A_1,\ldots ,A_n\in \mathcal{K}$ such that $\bigcup_{i=1}^nA_i^c=X$, by hypothesis.
Then $\bigcap_{i=1}^nA_i=\emptyset$, and
this is a contradiction to the fact that
$\mathcal K$ have  the finite intersection property.

\textit{Sufficiency.} 
Let $\mathcal{B}\subseteq \mathscr{A}$ such that $X=\bigcup \mathcal{B}$. Suppose that for every finite subset $\mathcal{A}$ of  $\mathcal{B}$, $X\not=\bigcup \mathcal{A}$, then $\mathcal K:=\{B^c:B\in  \mathcal{B}\}\subseteq \mathscr{A}$   have the finite intersection property, but $\bigcap\mathcal{K}=\emptyset$, and
this is a contradiction to statement (2).
Hence there exists a  finite subset $\mathcal{A}$ of  $\mathcal{B}$ such that
$X=\bigcup \mathcal{A}$.
Therefore, $(X, \mathscr{A})$ is a compact measurable space.
\end{proof}

Our next result shows that
compact measurable spaces admit a simple characterization in terms of fixed ideals
and fixed $Z_{\mathscr{A}}$-filters.
\begin{proposition}  \label{M95}
Let $(X, \mathscr{A})$ be a measurable space. The following statements are equivalent.
\begin{enumerate}
\item[{\rm (1)}]   $(X, \mathscr{A})$ is compact.
\item[{\rm (2)}]  Every proper ideal in $ M(X)$ is fixed.
\item[{\rm (3)}]  Every maximal ideal in $ M(X)$ is fixed.
\item[{\rm (4)}] Every $Z_{\mathscr{A}}$-filter in $\mathscr{A}$ is fixed.
\item[{\rm (5)}] Every $Z_{\mathscr{A}}$-ultrafilter in $\mathscr{A}$ is fixed.
\end{enumerate}
\end{proposition}
\begin{proof}
The equivalence of (2) with (4)  and (3) with (5)  are evident, by Propositions \ref{M30} and \ref{M35}.

(1)$\Rightarrow$(2).
Let $I$ be a proper ideal in $ M(X)$.
Then, by Proposition \ref{M30}, $Z[I]$ is  $Z_{\mathscr{A}}$-filter in $\mathscr{A}$.
Since  $Z[I]$ have  the finite intersection property, we conclude from Proposition \ref{M80} that  $\bigcap Z[I]\not=\emptyset$,
which follows that $I$  is a fixed ideal.

(2)$\Rightarrow$(3).
 Evident.

(3)$\Rightarrow$(1).
Suppose that $\mathcal K\subseteq \mathscr{A}$
  have the finite intersection property.
Then, by Remark \ref{M80},
there exists a $Z_{\mathscr{A}}$-altrafilter
$\mathcal{F}$ such that $\mathcal K\subseteq \mathcal{F}$.
Therefore, by Proposition \ref{M35},
$\emptyset\not=\bigcap\mathcal{F}\subseteq \bigcap\mathcal{K}$ and we conclude from
  Proposition \ref{M85} that $(X, \mathscr{A})$ is compact.
\end{proof}
\section{$T$-measurable}  \label{5}
In this section, we show that
 in the study of rings of measurable functions
 on a measurable space
there is no need to deal with measurable spaces that are not $T$-measurable.
\begin{definition}\label{M260}
Let $X$  be an abstract set, and consider an arbitrary
subfamily $C $ of $\mathbb{R}^X$.
The  {\it weak measurable space induced by $C $} on $X$
is defined
to be the smallest $\sigma$-algebra in $X$ such that all functions in $C $ are
measurable.
\end{definition}
Let $C \subseteq\mathbb{R}^X$ and
$\mathcal{A}:=\{f^{-1}(O): f\in C,\,  O\in \mathfrak{O}(\mathbb{R})\}$.
Then $(X,  <\mathcal{A}>)$ is the weak measurable space induced by $C$ on $X$.
\begin{lemma}\label{M270}
Let $(X,\mathscr{A})$  be a measurable space, $Y\not=\emptyset$ and $\mathcal{A}\subseteq \mathcal{P}(Y)$.
Suppose that $f:X\rightarrow Y$  is a function and
 $<\mathcal{A}>$ is the   $\sigma$-algebra generated by $\mathcal{A}$ on $Y$.
Then
$$\{f^{-1}(A) : A\in \mathcal{A}\}
\subseteq
\mathscr{A}$$
 if and only if
$$\{f^{-1}(B) : B\in <\mathcal{A}>\}
\subseteq
\mathscr{A}.$$
\end{lemma}
\begin{proof}
It is evident, because $f^{-1}(\bigcup_{\lambda\in \Lambda}B_{\lambda
})=\bigcup_{\lambda\in \Lambda}f^{-1}(B_{\lambda
})$ , $f^{-1}(\bigcap_{\lambda\in \Lambda}B_{\lambda
})=\bigcap_{\lambda\in \Lambda}f^{-1}(B_{\lambda
})$  and
$f^{-1}(B^{c})
=(f^{-1}(B))^{c}$,
for every $\{B_{\lambda
}\}_{\lambda\in \Lambda}\subseteq \mathcal{P}(Y)$
 and every $B\in  \mathcal{P}(Y)$.
\end{proof}

\begin{proposition} \label{M280}
Let $(X,\mathscr{A})$ be a measurable space,
$Y\not=\emptyset $ and $C\subseteq
\mathbb{R}^Y$.
Let $\mathscr{A}^{\prime}$ be
 the weak measurable space induced by $C$ on $Y$
 and  $f:X\rightarrow Y$ be a function.
 Then for every  $g\in C$, $g\circ f\in M(X)$
 if and only if
 for every $A\in\mathscr{A}^{\prime}$,
$f^{-1}(A)\in\mathscr{A}$.
\end{proposition}
\begin{proof}
\textit{Necessity.}
Let $O\in \mathfrak{O}(\mathbb{R})$,
then $f^{-1}(g^{-1}(O))\in \mathscr{A}$.
Therefore, By Definition \ref{M260} and
Lemma \ref{M270},  $f^{-1}(A)\in\mathscr{A}$
for every $A\in\mathscr{A}^{\prime}$.

\textit{Sufficiency.}
  Let  $O\in \mathfrak{O}(\mathbb{R})$,
then $g^{-1}(O) \in\mathscr{A}^{\prime}$,
which follows that $f^{-1}(g^{-1}(O))\in \mathscr{A}$. Therefore, $g\circ f\in M(X)$ for every  $g\in C$.
 \end{proof}

\begin{definition}   \label{M100} M100
A measurable space $(X, \mathscr{A})$ is said to be  {\it $T$-measurable}  if whenever $x$ and $y$ are distinct points in $X$, there is a measurable set containing one
and not the other (see \cite{Willard1970}).
\end{definition}
\begin{lemma}  \label{M285}
Let $(X,\mathscr{A})$ be a measurable space.
Define $x \sim x^{\prime}$ in $X$ to mean that $f(x) = f(x^{\prime})$ for every $f\in M(X)$.
Then the following statements hold.
\begin{enumerate}
\item[{\rm (1)}] The relation $\sim$ is an equivalence relation.
\item[{\rm (2)}]  If $  A\in \mathscr{A}$ then $A=\bigcup_{x\in A}[x]_{\sim}$.
\item[{\rm (3)}]  With each $f \in M(X)$), associate a function $h_f \in\mathbb{R}^{X/\sim}$ given by
$h_f([x]_{\sim})=f(x)$.
If  $\mathscr{A}_{X/\sim}$ is the
  weak measurable space induced by $\{h_f : f \in M(X)\} $ on ${X/\sim}$,
  then $(X/\sim,\mathscr{A}_{X/\sim})$ is
   $T$-measurable  space.
\end{enumerate}
\end{lemma}
\begin{proof}
(1).  It is evident.

(2). Consider  $x\in A\in \mathscr{A}$,
then, by Proposition \ref{M15},
there exists $f\in  M(X)$ such that $z(f)=A$.
 If $y\in [x]$, then $g(x)=g(y)$ for all $g\in
 {M}(X)$, which follows that  $f(y)=f(x)=0$ and hence $y\in z(f)=A$.

 (3). By the statement (2),  $(X/\sim,\mathscr{A}_{X/\sim})$ is
   measurable space.
Consider  $[x], [x^{\prime}]\in X/\sim$ with $[x]\not=[x^{\prime}]$.
 Then there exists $f\in M(X)$ such that $f(x)\not=f(x^{\prime})$, which follows that $h_f([x])\not=h_f([x^{\prime}])$.
 Consider $$r:=\dfrac{|h_f([x])-h_f([x^{\prime}])|}{3}.$$
  Hence, by Definition \ref{M260},
  $$[x]\in h_f^{-1}(h_f([x])-r,h_f([x])+r) \in\mathscr{A}_{X/\sim},$$
 $$[x^{\prime}]\in h_f^{-1}(h_f([x^{\prime}])-r,h_f([x^{\prime}])+r)\in\mathscr{A}_{X/\sim}$$
 and
$$ h_f^{-1}(h_f([x])-r,h_f([x])+r)\cap h_f^{-1}(h_f([x^{\prime}])-r,h_f([x^{\prime}])+r)=\emptyset.$$
 Therefore $(X/\sim,\mathscr{A}_{X/\sim})$ is
 $T$-measurable space.
 \end{proof}
The next proposition eliminates any reason for considering rings
of real measurable functions on other than $T$-measurable spaces.
\begin{proposition}  \label{M290}
For every measurable space $(X,\mathscr{A})$ there
is a $T$-measurable space
 $(Y,\mathscr{A}^{\prime})$
 and a onto function  $\theta:X\rightarrow Y$
  such that
   $\eta: M(Y)\rightarrow M(X)$
 given by  $g\mapsto g\circ\theta$
 is an isomorphism
  and    the following statements hold.
\begin{enumerate}
\item[{\rm (1)}]  For every
 $A\in\mathscr{A}$, $\theta(A)\in \mathscr{A}^{\prime}$.
\item[{\rm (2)}]  For every
 $B\in\mathscr{A}^{\prime}$, $\theta^{-1}(B)\in \mathscr{A}$.
\end{enumerate}
  \end{proposition}
\begin{proof}
Suppose $\sim$ is the same equivalence relation
Lemma \ref{M285}.
We put $Y:=X/\sim$ and
$\mathscr{A}^{\prime}:=\mathscr{A}_{X/\sim}$.
Then, by Lemma \ref{M285},
$(Y,\mathscr{A}^{\prime})$
is a  $T$--measurable space.
Define $\theta:X\rightarrow Y$ by
$\theta(x)=[x]_{\sim}$.
Hence $\theta$ is a onto function and
$h_f\circ\theta=f$ for any $f\in M(X)$, where $h_f$ is the same function
Lemma \ref{M285}.
 Then $\eta$  is onto function and we have
 $$\eta(g_1+g_2)(x)
 =((g_1+g_2)\circ\theta)(x)
= (g_1\circ\theta)(x) +(g_2\circ\theta)(x)
= (\eta(g_1)+\eta(g_2))(x)$$
and
 $$\eta(g_1  g_2)(x)
 =((g_1  g_2)\circ\theta)(x)
= (g_1 \circ\theta)(x) (g_2\circ\theta)(x)
=(\eta(g_1) \eta(g_2))(x),$$
for every $g_1,  g_2\in M(Y)$
and every $x\in X$.
Since $\theta$ is onto, we infer that
 $$ker(\eta)
 =\{g\in M(Y): \forall x\in X, g(\theta(x))=\{0\}\}
=\{g\in M(X):g(Y)=\{0\}\}
=\{0\},$$
 which follows that $\eta$ is an isomorphism, i.e.,
 $M(X)\cong  M(Y)$ as rings.

 If $A\in \mathscr{A}$  then, by Proposition \ref{M15},
   there is $f\in M(X)$ such that $A=z(f)$,
   which follows that
$$z(h_f)=\{[x]:h_f([x])=0\}=\{[x]:f(x)=0\}=\{[x]:x\in z(f)\}
=\theta( z(f)).$$
Therefore, by Proposition \ref{M15}, $\theta (A)\in \mathscr{A}'$.
Thus, the statement (1) holds.

Since $\mathscr{A}^{\prime}$ is the
  weak measurable space induced by $\{h_f : f \in M(X)\} $ on $Y$ and for every $f \in M(X)$, $h_f\circ\theta=f$,
  we conclude from Lemma \ref{M280}  that
  for every $B\in\mathscr{A}^{\prime}$,
$\theta^{-1}(B)\in \mathscr{A}$.
Thus, the statement (2) holds.
\end{proof}

As a consequence of the foregoing theorem,
algebraic or lattice properties that hold for all $M(X)$,
 with $X$ \, $T$-measurable space,
 hold just as well for all $M(X)$, with $X$  arbitrary.

Now, we give an algebraic characterization
$T$-measurable spaces in terms of
 maximal ideals.
\begin{proposition}   \label{M105}
Let $(X, \mathscr{A})$ be a measurable space. Then the following statements are equivalent.
\begin{enumerate}
\item[{\rm (1)}]  The measurable space $(X, \mathscr{A})$ is a  $T$-measurable space.
\item[{\rm (2)}]  If $x$ and $y$ are distinct points in $X$, then $M_x\not=M_x$.
\item[{\rm (3)}]  For every maximal ideal $M$ in $ M(X)$, $|\bigcap_{f\in M}Z(f)|\leq 1$.
\end{enumerate}
\end{proposition}
\begin{proof}
(1)$\Rightarrow$(2).
Assume that  $x$ and $y$ are distinct points in $X$. By  hypothesis, there exists a  measurable
set $A$ in $X$ such that $x\in A$ and $y\not\in A$. By Remark \ref{chi}, $\chi_{_A}\in M_y\setminus M_x$.

(2)$\Rightarrow$(3).
Let $M$ be a maximal ideal  in $ M(X)$ with $|\bigcap_{f\in M}Z(f)|\geq 2$.
If  $x$ and $y$ are distinct points in $\bigcap_{f\in M}Z(f)$,
then $M\subseteq M_x$ and $M\subseteq M_y$.
Since $M$ is maximal, we conclude from
Proposition \ref{M45} that $M_y=M=M_x$,
and this is a contradiction.

(3)$\Rightarrow$(1).
Assume that  $x$ and $y$ are distinct points in $X$. Then, Proposition \ref{M45},
 $\bigcap_{f\in M_x}Z(f)= \{x\}$ and $\bigcap_{f\in M_y}Z(f)= \{y\}$, which follows that there exists  an $f$ in $M_x\setminus M_y$.
Since $Z(f)\in   \mathscr{A}$, $x\in Z(f)$ and
$y\not\in Z(f)$, we infer that  $(X, \mathscr{A})$ is a  $T$-measurable space.
\end{proof}
\begin{corollary}  \label{X-P=1}
Let $(X, \mathscr{A})$ be a $T$-measurable space.
 For every element $P$ of   $\mathscr{A}$,
  $P$ is a prime element of
  $\mathscr{A}$   if and only if \, $|X\setminus P|=1$.
\end{corollary}
\begin{proof}
\textit{Necessity.}
By Propositions  \ref{fixmax} and \ref{M105},
$|X\setminus P|=|\bigcap_{f\in M_P}  z(f)|=1$.

\textit{Sufficiency.}  It is clear.
\end{proof}
We recall that a ring
$R$ is called a Gelfand ring or a $PM$-ring if each of its proper prime ideals is contained in a unique maximal ideal.
 The following  proposition shows that
 $M(X)$ is  a Gelfand ring.
\begin{proposition}   \label{M115}
Let $(X, \mathscr{A})$ be a measurable space.
Every prime ideal in $ M(X)$ is contained in a unique
maximal ideal.
\end{proposition}
\begin{proof}
Let $P$ be a prime ideal.
 We know that every ideal is contained in at least one maximal
ideal.
If $M$ and $M'$ are distinct maximal ideals such that $P\subseteq M\cap M'$,
then, by Proposition \ref{M110}, $M\cap M'$ is a prime ideal, since $M$ and $M'$ are $Z_{\mathscr{A}}$-ideals.
This is a contradiction to the fact that
 $M\cap M'$ is not prime.
\end{proof}
The following proposition shows that
$T$-measurable spaces have a nice  characterization in terms of prime ideals.
\begin{proposition}   \label{M120} M120
Let $(X, \mathscr{A})$ be a measurable space.
Then
the measurable space $(X, \mathscr{A})$ is a  $T$-measurable space
if and only if
for every prime ideal $P$ in $ M(X)$, $|\bigcap_{f\in P}Z(f)|\leq 1$.
%
\end{proposition}
\begin{proof}
\textit{Necessity.}
Let $P$ is a prime ideal  in $ M(X)$ such that $x,y\in \bigcap_{f\in P}Z(f)$ with $x\not=y$.
Then, by hypothesis and
Proposition \ref{M105},  $M_x$ and $M_y$ are distinct maximal ideals such that $P\subseteq M_x\cap M_y$.
This is a contradiction to the fact that
every prime ideal in $ M(X)$ is contained in a unique
maximal ideal.

\textit{Sufficiency.}
Let $M$ be a maximal ideal  in $ M(X)$, then  $M$ be a prime ideal  in $ M(X)$, which follows that
$|\bigcap_{f\in M}Z(f)|\leq 1$, by hypothesis.
Therefore,  by Proposition \ref{M105},
  $(X, \mathscr{A})$ is a  $T$-measurable space.
\end{proof}
\begin{definition} \label{M125}
{\rm \cite{EbrahimiMahmoudi1996}}
  Let $L$ be a distributive lattice and $F$ be a filter  of $L$. The filter $F$ is called   {\it prime filter}
 if  $I  \not= L$ and  $x \vee y  \in  F$
implies $x  \in  F$ or $y  \in  F$.
\end{definition}
In the following proposition, we study
relations between prime ideals and prime  $Z_{\mathscr{A}}$-filters.
\begin{proposition} \label{M130}  M130
Let $(X, \mathscr{A})$ be a measurable space.
In $ M(X)$,  the following statements hold.
\begin{enumerate}
\item[{\rm (1)}]  If $P$ is a prime ideal in $ M(X)$, then $Z[P]$ is a prime  $Z_{\mathscr{A}}$-filter on $X$.
\item[{\rm (2)}]  If $ \mathcal F$ is a prime  $Z_{\mathscr{A}}$-filter on $X$, then $Z^{-1}[\mathcal F]$ is a prime ideal in $ M(X)$.

The mapping $Z$ is one-one from the set of all prime ideals in $ M(X)$ onto
the set of all  prime $Z_{\mathscr{A}}$-filters on $X$.
\end{enumerate}
\end{proposition}
\begin{proof}
 (1). Consider $f,g\in  M(X)$ with $Z(fg)=Z(f)\cup Z(g)\in Z[P]$. Since $P_z:=Z^{-1}[Z[P]]$ is a $Z_{\mathscr{A}}$-ideal in  $ M(X)$ and $Z(fg)\in Z[P]= Z[P_z]$, we conclude that $fg\in P_z$.
 From $P\subseteq P_z$ and
 Proposition \ref{M110}, we infer that $P_z$ is a prime ideal in $ M(X)$, which follows that $f\in P_z$ or  $g\in P_z$.
 Hence $Z(f)\in Z[P_z]= Z[P]$ or $Z(g)\in Z[P_z]= Z[P]$ and, by Proposition \ref{M30},  the proof is now complete.

 (2). Consider  $f,g\in  M(X)$ with $fg\in Z^{-1}[\mathcal F]$,
 then $Z(f)\cup Z(g)=Z(fg)\in F$, which follows that $Z(f) \in F$ or $ Z(g) \in F$, by hypothesis.
Hence  $f \in  Z^{-1}[\mathcal F]$ or $ g \in  Z^{-1}[\mathcal F]$
 and, by Proposition \ref{M30},  the proof is now complete.
\end{proof}

\section{On compact measurable spaces}  \label{6}
In this section,   we study the third subject.
We begin with the following lemma.
\begin{lemma}   \label{M200-1}
Let $(X, \mathscr{A})$ be a compact measurable space.
Suppose $\theta$ and $(Y,\mathscr{A}^{\prime})$
 are the same symptoms Proposition \ref{M290}.
  Then the following statements  hold.
\begin{enumerate}
\item[{\rm (1)}] If $A$  is a compact element  of $ \mathscr{A}$, then
$\theta(A)$ is a compact element  of $\mathscr{A}^{\prime}$.
\item[{\rm (2)}] If $B\in \mathscr{A}'$   is a compact element  of $ \mathscr{A}^{\prime}$, then
$\theta^{-1}(B)$  is a compact element  of $\mathscr{A} $.
\end{enumerate}
\end{lemma}
\begin{proof}
Let $A$   be  a compact  element of $ \mathscr{A}$
 and
$\{B_{\lambda}: \lambda\in\Lambda\} \subseteq\mathscr{A}'$  such that $\theta(A)\subseteq\bigcup_{\lambda\in\Lambda}
B_{\lambda}$.
From
 $A\subseteq\theta^{-1}(\bigcup_{\lambda\in\Lambda}
B_{\lambda})=\bigcup_{\lambda\in\Lambda} \theta^{-1}(B_{\lambda})$,
we conclude that there exists a finite subset $\Lambda^{\prime}$ of  $\Lambda$ such that
 $A\subseteq\bigcup_{\lambda\in\Lambda^{\prime}} \theta^{-1}(B_{\lambda})$, and hence
$\theta(A)\subseteq\theta(\bigcup_{\lambda\in\Lambda'}
\theta^{-1}(B_{\lambda}))=
\bigcup_{\lambda\in\Lambda'} B_{\lambda}$.
Therefore,
 $\theta(A)$ is a  compact element of
 $ \mathscr{A}^{\prime}$.
 Thus, the statement (1) holds.

Let $B\in \mathscr{A}'$  be a compact element  of $ \mathscr{A}^{\prime}$ and
 $\{A_{\lambda}:\lambda\in\Lambda\}\subseteq\mathscr{A}$
such that  $\theta^{-1}(B)\subseteq \bigcup_{\lambda\in\Lambda}A_{\lambda}$.
Since  $\theta$ is the onto function,
we infer that
$$B=\theta(\theta^{-1}(B))\subseteq \theta(\bigcup_{\lambda\in\Lambda}A_{\lambda})=\bigcup_{\lambda\in\Lambda}\theta (A_{\lambda}),$$
which follows that
 there is a finite subset $\Lambda^{\prime}$
 of $\Lambda$ such that $B\subseteq \bigcup_{\lambda\in\Lambda'}\theta (A_{\lambda})$,
 in other words
 $\theta^{-1}(B)\subseteq \bigcup_{\lambda\in\Lambda^{\prime}} A_{\lambda} $. Therefore,
 $\theta^{-1}(B)$ is  a compact element of $ \mathscr{A}$.
   Thus, the statement (2) holds.
\end{proof}
\begin{corollary}  \label{M295} M295
For every compact measurable space $(X,\mathscr{A})$ there
is a   compact $T$-measurable  space
 $(Y,\mathscr{A}^{\prime})$
  such that $$M(X)\cong  M(Y), $$
  as rings.
\end{corollary}
\begin{proof}
It is evident, by Proposition \ref{M290} and Lemma  \ref{M200-1}.
\end{proof}

\begin{remark}  \label{M200} M200
Let $(X, \mathscr{A})$ be a compact
measurable space.
Consider $A\in \mathscr{A}$  and
$\{ B_{\lambda} \}_{ \lambda \in \Lambda }\subseteq  \mathscr{A}$ with
$A\subseteq \bigcup_{ \lambda \in \Lambda }B_{\lambda}$.
From $A^c\in  \mathscr{A}$ and
$X=A^c\cup  \bigcup_{ \lambda \in \Lambda }B_{\lambda}$, we infer that there exists
 a finite subset $\Lambda_0$ of $\Lambda$ such that
$X=A^c\cup  \bigcup_{ \lambda \in \Lambda_0 }B_{\lambda}$,
which follows that $A\subseteq \bigcup_{ \lambda \in \Lambda_0 }B_{\lambda}$.
Therefore,  for every $A\in  \mathscr{A} $,  $A$ is  a compact element of $\mathscr{A}$.
\end{remark}

\begin{corollary} \label{M205} M205
 Let $(X, \mathscr{A})$ be a measurable space and $A,B \in  \mathscr{A}$.
 The following statements hold.
\begin{enumerate}
\item[{\rm (1)}]  If $ A$ or $B$ is  a compact element of $\mathscr{A}$, then  $A\cap B $ is a compact element of $\mathscr{A}$.

\item[{\rm (2)}]  If $ A$ and $B$ are    compact elements of $\mathscr{A}$, then
$A\cup B $ is a compact element of $\mathscr{A}$.
\item[{\rm (3)}]  $A $ is a compact element of $\mathscr{A}$  if and only
 if $\uparrow A^{c}$  is a compact lattice.
 \item[{\rm (4)}]  If  $A $ is a compact element of $\mathscr{A}$
 then $B $ is a compact element of $\mathscr{A}$, for every
 $B\in ‎\downarrow A$.
\end{enumerate}
\end{corollary}
\begin{proof}
It is evident.
\end{proof}
Let $\max( M(X))$ be  the set of all maximal ideals of  $M(X)$.
Throughout this paper, we put
$$
\mathcal{F}(f):=\{M\in  \max(M(X)):f\in M \},
$$
for every $f\in M(X)$.
Hence,
$\mathcal{F}(\mathbf1)=\emptyset$
and  $\mathcal{F}(\mathbf0)= \max(M(X))$.
Also, by Propositions \ref{M95} and \ref{M105},
 if
$(X, \mathscr{A})$ is a   compact $T$-measurable space, then
$$
\mathcal{F}(f)= \{M_x:x\in z(f)\},
$$
for every $f\in M(X)$.
\begin{proposition} \label{M210} M210
For every   compact $T$-measurable  space $(X, \mathscr{A})$,  $$\big(\max(M(X)), \left\{\mathcal{F}(f):f\in M(X) \right\}\big)$$ is a $T$-measurable space.
\end{proposition}
\begin{proof}
Consider $f\in M(X)$
and
$\{f_n\}_{n\in\mathbb N}\subseteq  M(X)$. Hence,
\[
\mathcal{F}(f)^{c}
=\{M_x :  x\in {coz (f)} \}
=\{M_x :  x\in z(\chi_ {_ {z(f)}}) \}
=\mathcal{F}(\chi_{_{z(f)}})
\]
and
\[
\bigcup_{n\in \mathbb{N}} \mathcal{F}(f_n)=
\bigcup_{n\in \mathbb{N}}\{M_x:x\in z(f_n) \}
=\{M_x: x\in \bigcup_{n\in \mathbb{N}} z(f_n)\}
=\mathcal{F}(\chi_{_{(\bigcap_{n\in \mathbb{N}}}
{ \coz(f_n)})}).
\]
Also, if  $M_1$ and $M_2$ are distinct points in $\max(M(X))$, then there exists an element $f\in M(X)$
such that $f\in M_1\setminus M_2$, which follows that
$M_1\in  \mathcal{F}(f)$ and $M_2\not\in  \mathcal{F}(f)$. Hence  $\big(\max(M(X)), \left\{\mathcal{F}(f):f\in M(X) \right\}\big)$ is a $T$-measurable space.
\end{proof}

\begin{definition}  \label{M 212} M212
Let $(X_1,\mathscr{A}_1)$  and $(X_2,\mathscr{A}_2)$ are
measurable spaces.
We say
 that $(X_1,\mathscr{A}_1)$  and $(X_2,\mathscr{A}_2)$ are  {\it homeomorphic},
 provided that
 there exists a  one to one
and onto function
$f:X_1\rightarrow X_2$ such that
$$A\in \mathscr{A}_1  \Leftrightarrow f(A)\in\mathscr{A}_2,$$ for every $A\in \mathscr{A}_1$.
\end{definition}
If we denote "$(X_1,\mathscr{A}_1)$ is homeomorphic with $(X_2,\mathscr{A}_2)$"
by $X\cong Y$, then the relationship $ \cong $
is an equivalence relation on any set
of measurable spaces.

\begin{proposition} \label{M215} M215
For every   compact $T$-measurable  space $ (X, \mathscr{A})$,  $$X\cong \max(M(X))$$ as measure spaces.
\end{proposition}
\begin{proof}
We define
\[
\begin{array}{rll}
\varphi:X
&\longrightarrow &\max(M(X))\\[1mm]
x &\longmapsto &M_x.
\end{array}
\]
By  Propositions \ref{M95} and \ref{M105},
$\varphi$ is a one-one correspondence
and also, for every $f\in M(X)$, we have
$$\varphi[z(f)]=\{M_x:x\in z(f)\}=\mathcal{F}{(f)}$$
and
$$\varphi^{-1}(\mathcal{F}{(f)})=\{x\in X:x\in z(f)\}=z(f).$$
Therefore, by Proposition \ref{M15},
$X\cong \max(M(X))$ as measure spaces.
\end{proof}
The measurable space is defined in Proposition
\ref{M215},  is called the Stone measure on $\max( M(X))$.

\begin{corollary} \label{M220} M220
If $(X, \mathscr{A})$ and  $(Y, \mathfrak{M^{\prime}})$ are two compact
$T$-measurable spaces, then  $X\cong Y$ as measurable spaces if and only if
 $M(X)\cong  M (Y)$ as rings.
\end{corollary}
\begin{proof}
By  Proposition \ref{M215}, it is obvious.
\end{proof}
\section*{Acknowledgements} 


We would like to thank 
 from professor  Gader Sadighi 
 and  
 professor   Ali Akbar Arefi Jamal, 
 Hakim Sabzevari University, 
  for their detailed comments 
  and suggestions for the manuscript.

\end{document}